\documentclass[12pt,english]{article}
\usepackage[T1]{fontenc}
\usepackage[latin9]{inputenc}
\usepackage[active]{srcltx}
\usepackage{mathrsfs}
\usepackage{amsmath}
\usepackage{amssymb}

\makeatletter
\usepackage[T1]{fontenc}
\usepackage[latin9]{inputenc}
\usepackage{color}
\usepackage{amsmath}
\usepackage{graphicx}
\usepackage{amsfonts}
\usepackage{a4wide}
\usepackage{babel}%

\usepackage{mathrsfs}
\usepackage[active]{srcltx}

\providecommand{\U}[1]{\protect\rule{.1in}{.1in}}

\newtheorem{theorem}{Theorem}

\newtheorem{corollary}[theorem]{Corollary}

\newtheorem{definition}[theorem]{Definition}

\newtheorem{lemma}[theorem]{Lemma}

\newenvironment{proof}[1][Proof]{\noindent\textbf{#1.} }{\ \rule{0.5em}{0.5em}}

\usepackage[usenames,dvipsnames]{pstricks}
\usepackage{epsfig}
\usepackage{pst-grad} 
\usepackage{pst-plot} 

\makeatother

\usepackage{babel}
\begin{document}

\title{The local equicontinuity of a maximal monotone operator}

\author{M.D. Voisei}

\date{{}}
\maketitle
\begin{abstract}
The local equicontinuity of an operator $T:X\rightrightarrows X^{*}$
with proper Fitzpatrick function $\varphi_{T}$ and defined in a barreled
locally convex space $X$ has been shown to hold on the algebraic
interior of $\operatorname*{Pr}_{X}(\operatorname*{dom}\varphi_{T})$)%
\footnote{see \cite[Theorem\ 4]{MR3252437}%
}. The current note presents direct consequences of the aforementioned
result with regard to the local equicontinuity of a maximal monotone
operator defined in a barreled locally convex space. 
\end{abstract}

\section{Introduction and notations}

Throughout this paper, if not otherwise explicitly mentioned, $(X,\tau)$
is a non-trivial (that is, $X\neq\{0\}$) real Hausdorff separated
locally convex space (LCS for short), $X^{\ast}$ is its topological
dual endowed with the weak-star topology $w^{\ast}$, the weak topology
on $X$ is denoted by $w$, and $(X^{*},w^{*})^{*}$ is identified
with $X$. The class of neighborhoods of $x\in X$ in $(X,\tau)$
is denoted by $\mathscr{V}_{\tau}(x)$. 

The \emph{duality product} or \emph{coupling }of $X\times X^{\ast}$
is denoted by $\left\langle x,x^{\ast}\right\rangle :=x^{\ast}(x)=:c(x,x^{\ast})$,
for $x\in X$, $x^{\ast}\in X^{\ast}$. With respect to the dual system
$(X,X^{*})$, the \emph{polar of} $A\subset X$ is $A^{\circ}:=\{x^{*}\in X^{*}\mid|\langle x,x^{*}\rangle|\le1,\ \forall x\in A\}$.

A set $B\subset X^{*}$ is ($\tau-$)\emph{equicontinuous }if for
every $\epsilon>0$ there is $V_{\epsilon}\in\mathscr{V}_{\tau}(0)$
such that, for every $x^{*}\in B$, $x^{*}(V_{\epsilon})\subset(-\epsilon,\epsilon)$,
or, equivalently, $B$ is contained in the polar $V^{\circ}$ of some
(symmetric) $V\in\mathscr{V}_{\tau}(0)$.

A multi-function $T:X\rightrightarrows X^{*}$ is ($\tau-$)locally
equicontinuous%
\footnote{``locally bounded'' is the notion used in the previous papers on
the subject (see e.g. \cite[Definition\ 1]{MR3252437}, \cite[p. 397]{MR0253014}),
mainly because in the context of barreled spaces, weak-star bounded
and equicontinuous subsets of the dual coincide. %
} at $x_{0}\in X$ if there exists $U\in\mathscr{V}_{\tau}(x_{0})$
such that $T(U):=\cup_{x\in U}Tx$ is a ($\tau-$)equicontinuous subset
of $X^{*}$; ($\tau-$)locally equicontinuous on $S\subset X$ if
$T$ is ($\tau-$)locally equicontinuous at every $x\in S$. The ($\tau-$)local
equicontinuity of $T:X\rightrightarrows X^{*}$ is interesting only
at $x_{0}\in\overline{D(T)}^{\tau}$ ($\tau-$closure) since, for
every $x_{0}\not\in\overline{D(T)}^{\tau}$, $T(U)$ is void for a
certain $U\in\mathscr{V}_{\tau}(x_{0})$. Here $\operatorname*{Graph}T=\{(x,x^{*})\in X\times X^{*}\mid x^{*}\in Tx\}$
is the graph of $T$, $D(T)=\operatorname*{Pr}_{X}(\operatorname*{Graph}T)$
stands for the domain of $T$, where $\operatorname*{Pr}_{X}$ denotes
the projection of $X\times X^{*}$ onto $X$.

The main objective of this paper is to give a description of the local
equicontinuity set (see Definition \ref{les} below) of a maximal
monotone operator $T:X\rightrightarrows X^{*}$ ($T\in\mathfrak{M}(X)$
for short) defined in a barreled space $X$ in terms of its \emph{Fitzpatrick
function} $\varphi_{T}:X\times X^{*}\rightarrow\overline{\mathbb{R}}$
which is given by (see \cite{MR1009594})
\begin{equation}
\varphi_{T}(x,x^{*}):=\sup\{\langle x-a,a^{*}\rangle+\langle a,x^{*}\rangle\mid(a,a^{*})\in\operatorname*{Graph}T\},\ (x,x^{*})\in X\times X^{*}.\label{ff}
\end{equation}

As usual, given a LCS $(E,\mu)$ and $A\subset E$ we denote by ``$\operatorname*{conv}A$''
the \emph{convex hull} of $A$, ``$\operatorname*{cl}_{\mu}(A)=\overline{A}^{\mu}$''
the $\mu-$\emph{closure} of $A$, ``$\operatorname*{int}_{\mu}A$''
the $\mu-$\emph{topological interior }of $A$, ``$\operatorname*{core}A$''
the \emph{algebraic interior} of $A$. The use of the  $\mu-$notation
is enforced barring that the topology $\mu$ is clearly understood. 

When $X$ is a Banach space, Rockafellar showed in \cite[Theorem\ 1]{MR0253014}
that if $T\in\mathfrak{M}(X)$ has $\operatorname*{int}(\operatorname*{conv}D(T))\neq\emptyset$
then $D(T)$ is \emph{nearly-solid} in the sense that $\operatorname*{int}D(T)$
is non-empty, convex, whose closure is $\overline{D(T)}$, while $T$
is locally equicontinuous (equivalently said, locally bounded) at
each point of $\operatorname*{int}D(T)$ and unbounded at any boundary
point of $D(T)$. 

Our results extend the results of Rockafellar \cite{MR0253014} to
the framework of barreled LCS's or present new shorter arguments for
known ones (see Theorems \ref{ot}, \ref{ot-bn}, \ref{mnc}, \ref{msd}
below).

\section{The local equicontinuity set}

One of the main reasons for the usefulness of equicontinuity is that
it ensures the existence of a limit for the duality product on the
graph of an operator. More precisely, if $T:X\rightrightarrows X^{*}$
is locally equicontinuous at $x_{0}\in\overline{D(T)}^{\tau}$ then
for every net $\{(x_{i},x_{i}^{*})\}_{i}\subset\operatorname*{Graph}T$
with $x_{i}\overset{\tau}{\to}x_{0}$, $\{x_{i}^{*}\}_{i}$ is equicontinuous
thus, according to Bourbaki's Theorem, weak-star relatively compact
in $X^{*}$ and, at least on a subnet, $x_{i}^{*}\rightarrow x_{0}^{*}$
weak-star in $X^{*}$ and $\lim_{i}c(x_{i},x_{i}^{*})=c(x_{0},x_{0}^{*})$. 

\medskip

\begin{definition} \emph{\label{les}Given $(X,\tau)$ a LCS, for
every $T:X\rightrightarrows X^{*}$ we denote by
\[
\Omega_{T}^{(\tau)}:=\{x\in\overline{D(T)}^{\tau}\mid T\ {\rm is\ (\tau-)locally\ equicontinuous\ at}\ x\}
\]
the }(meaningful) ($\tau-$)local equicontinuity set of\emph{ $T$.}
\end{definition}

In the above notation, our local equicontinuity result \cite[Theorem\ 8]{MR3252437}
states that if the LCS $(X,\tau)$ is barreled then
\begin{equation}
\forall T:X\rightrightarrows X^{*},\ \operatorname*{core}\operatorname*{Pr}\,\!\!_{X}(\operatorname*{dom}\varphi_{T})\cap\overline{D(T)}^{\tau}\subset\Omega_{T}^{\tau}.\label{eq:lbg}
\end{equation}

\begin{theorem} \label{ot} If $(X,\tau)$ is a barreled LCS and
$T\in\mathfrak{M}(X)$ has $\overline{D(T)}^{\tau}$ convex then 
\begin{equation}
\operatorname*{int}\,\!_{\tau}\operatorname*{Pr}\,\!\!_{X}(\operatorname*{dom}\varphi_{T})=\operatorname*{core}\operatorname*{Pr}\,\!\!_{X}(\operatorname*{dom}\varphi_{T}).\label{eq:-2}
\end{equation}

The following are equivalent

\medskip

\emph{(i)} $\operatorname*{core}\operatorname*{Pr}\,\!\!_{X}(\operatorname*{dom}\varphi_{T})\neq\emptyset$,\emph{
(ii)} $\Omega_{T}^{\tau}\neq\emptyset$ and $\operatorname*{core}\overline{D(T)}^{\tau}\neq\emptyset$,
\emph{(iii)} $\operatorname*{int}_{\tau}D(T)\neq\emptyset$.

\medskip

In this case
\begin{equation}
\Omega_{T}^{\tau}=\operatorname*{core}\operatorname*{Pr}\,\!\!_{X}(\operatorname*{dom}\varphi_{T})=\operatorname*{int}\,\!_{\tau}D(T)\label{eq:-3}
\end{equation}
 and $D(T)$ is $\tau-$near\-ly-solid in the sense that $\operatorname*{int}_{\tau}D(T)=\operatorname*{int}_{\tau}\overline{D(T)}^{\tau}$
and $\overline{D(T)}^{\tau}=\operatorname*{cl}_{\tau}(\operatorname*{int}_{\tau}D(T))$.
\end{theorem}

In the proof of the Theorem \ref{ot} we use the following lemma.

\begin{lemma} \label{lost} Let $(X,\tau)$ be a LCS and $T:X\rightrightarrows X^{*}$.
Then 
\begin{equation}
\Pr\nolimits _{X}(\operatorname*{dom}\varphi_{T})\subset D(T^{+})\cup\overline{\operatorname*{conv}}^{\tau}(D(T)).\label{eq:}
\end{equation}
Here $T^{+}:X\rightrightarrows X^{*}$ is the operator associated
to the set of elements monotonically related to $T$, $\operatorname*{Graph}T^{+}=\{(x,x^{*})\in X\times X^{*}\mid\langle x-a,x^{*}-a^{*}\rangle\ge0,\ \forall(a,a^{*})\in\operatorname*{Graph}T\}$. 

If, in addition, $T\in\mathcal{M}(X)$, then 
\begin{equation}
\Pr\nolimits _{X}(\operatorname*{dom}\varphi_{T})\subset D(T^{+})\cup\overline{\operatorname*{conv}}^{\tau}(D(T))\subset\operatorname*{cl}\!\!\,_{\tau}\Pr\nolimits _{X}(\operatorname*{dom}\varphi_{T}).\label{lost-rec}
\end{equation}

If, in addition, $T\in\mathfrak{M}(X)$, then 
\begin{equation}
\operatorname*{cl}\!\,\!_{\tau}\operatorname*{Pr}\!\,\!_{X}(\operatorname*{dom}\varphi_{T})=\overline{\operatorname*{conv}}^{\tau}(D(T)).\label{eq:-1}
\end{equation}

If, in addition, $T\in\mathfrak{M}(X)$ and $\overline{D(T)}^{\tau}$
is convex then 
\begin{equation}
\operatorname*{cl}\!\,\!_{\tau}\operatorname*{Pr}\!\,\!_{X}(\operatorname*{dom}\varphi_{T})=\overline{D(T)}^{\tau}.\label{eq:-4}
\end{equation}
\end{lemma}

\begin{proof} Let $T:X\rightrightarrows X^{*}$ and $(x,x^{*})\in\operatorname*{dom}\varphi_{T}$
with $x\not\in\overline{\operatorname*{conv}}^{\tau}(D(T))$. Separate
to get $y^{*}\in X^{*}$ such that $\inf\{\langle x-a,y^{*}\rangle\mid a\in D(T)\}\ge(\varphi_{T}-c)(x,x^{*})$.
For every $(a,a^{*})\in T$ we have
\[
\langle x-a,x^{*}+y^{*}-a^{*}\rangle\ge(c-\varphi_{T})(x,x^{*})+\langle x-a,y^{*}\rangle\ge0,
\]
 i.e., $(x,x^{*}+y^{*})\in\operatorname*{Graph}T^{+}$ so $x\in D(T^{+})$.
Therefore the first inclusion in (\ref{lost-rec}) holds for every
$T:X\rightrightarrows X^{*}$ while the second inclusion in (\ref{lost-rec})
follows from $D(T)\subset\operatorname*{Pr}_{X}(\operatorname*{dom}\varphi_{T})$
and $\operatorname*{Graph}T^{+}\subset\operatorname*{dom}\varphi_{T}$. 

If, in addition, $T\in\mathfrak{M}(X)$, then $T=T^{+}$ and (\ref{lost-rec})
translates into $\operatorname*{cl}_{\tau}\operatorname*{Pr}_{X}(\operatorname*{dom}\varphi_{T})=\overline{\operatorname*{conv}}^{\tau}(D(T))$.
\end{proof}

\strut

\begin{proof}[Proof of Theorem 2] According to Lemma \ref{lost},
$\operatorname*{cl}\,\!\!_{\tau}(\operatorname*{Pr}\,\!\!_{X}(\operatorname*{dom}\varphi_{T}))=\overline{D(T)}^{\tau}$. 

Because $T\in\mathfrak{M}(X)$ we claim that 
\begin{equation}
\forall x\in\Omega_{T}^{\tau},\ \exists U\in\mathscr{V}_{\tau}(x),\ T(U)\ {\rm is\ equicontinuous}\ {\rm and}\ U\cap\overline{D(T)}^{\tau}\subset D(T).\label{cyber}
\end{equation}

In particular $\Omega_{T}^{\tau}\subset D(T)$. 

Indeed, for every $x\in\Omega_{T}^{\tau}$ take $U$ a $\tau-$open
neighborhood of $x$ such that $T(U)$ is ($\tau-$)equicontinuous;
in particular $T(U)$ is weak-star relatively compact. For every $y\in U\cap\overline{D(T)}^{\tau}$
there is a net $(y_{i})_{i\in I}\subset U\cap D(T)$ such that $y_{i}\overset{\tau}{\to}y$.
Take $y_{i}^{*}\in Ty_{i}$, $i\in I$. At least on a subnet $y_{i}^{*}\to y^{*}$
weakly-star in $X^{*}$. For every $(a,a^{*})\in T$, $\langle y_{i}-a,y_{i}^{*}-a^{*}\rangle\ge0$,
because $T\in\mathcal{M}(X)$. After we pass to limit, taking into
account that $(y_{i}^{*})$ is ($\tau-$)equicontinuous, we get $\langle y-a,y^{*}-a^{*}\rangle\ge0$,
for every $(a,a^{*})\in T$. This yields $(y,y^{*})\in T$ due to
the maximality of $T$; in particular $y\in D(T)$. Therefore $U\cap\overline{D(T)}^{\tau}\subset D(T)$. 

Whenever $\operatorname*{core}\overline{D(T)}^{\tau}$ is non-empty,
$\operatorname*{int}\,\!\!_{\tau}\overline{D(T)}^{\tau}=\operatorname*{core}\overline{D(T)}^{\tau}\neq\emptyset$,
since $(X,\tau)$ is barreled. This yields that $\overline{D(T)}^{\tau}$
is solid, so, $\overline{D(T)}^{\tau}=\operatorname*{cl}_{\tau}(\operatorname*{int}_{\tau}\overline{D(T)}^{\tau})$
(see e.g. \cite[Lemma\ p.\ 59]{MR0410335}). 

Whenever $\Omega_{T}^{\tau}\neq\emptyset$ and $\operatorname*{core}\overline{D(T)}^{\tau}\neq\emptyset$,
from (\ref{cyber}), $D(T)$ contains the non-empty open set $U\cap\operatorname*{int}\,\!\!_{\tau}\overline{D(T)}^{\tau}$;
$\operatorname*{int}\,\!\!_{\tau}D(T)\neq\emptyset$ and so $\operatorname*{int}\,\!\!_{\tau}\operatorname*{Pr}_{X}(\operatorname*{dom}\varphi_{T})\neq\emptyset$
which also validates (\ref{eq:-2}).

From the above considerations and 
\[
\operatorname*{int}\,\!\!_{\tau}D(T)\subset\operatorname*{core}\operatorname*{Pr}\,\!\!_{X}(\operatorname*{dom}\varphi_{T})\subset\Omega_{T}^{\tau}\cap\operatorname*{core}\overline{D(T)}^{\tau}
\]
 we see that (iii) $\Rightarrow$ (i)$\Rightarrow$ (ii) $\Rightarrow$
(iii).

In this case, since $\operatorname*{Pr}_{X}(\operatorname*{dom}\varphi_{T})$
is convex, according to \cite[Lemma\ p.\ 59]{MR0410335},  
\[
\operatorname*{core}\operatorname*{Pr}\,\!\!_{X}(\operatorname*{dom}\varphi_{T})=\operatorname*{int}\,\!\!_{\tau}\operatorname*{Pr}\,\!\!_{X}(\operatorname*{dom}\varphi_{T})=\operatorname*{int}\,\!\!_{\tau}(\operatorname*{cl}\,\!\!_{\tau}(\operatorname*{Pr}\,\!\!_{X}(\operatorname*{dom}\varphi_{T})))=\operatorname*{int}\,\!\!_{\tau}\overline{D(T)}^{\tau}.
\]
Hence 
\[
\operatorname*{int}\,\!\!_{\tau}D(T)\subset\operatorname*{core}D(T)\subset\operatorname*{core}\operatorname*{Pr}\,\!\!_{X}(\operatorname*{dom}\varphi_{T})=\operatorname*{int}\,\!\!_{\tau}\operatorname*{Pr}\,\!\!_{X}(\operatorname*{dom}\varphi_{T})\subset\Omega_{T}^{\tau}\subset D(T)
\]
followed by  $\operatorname*{int}\,\!\!_{\tau}D(T)=\operatorname*{core}D(T)=\operatorname*{core}\operatorname*{Pr}\,\!\!_{X}(\operatorname*{dom}\varphi_{T})=\operatorname*{int}\,\!\!_{\tau}\overline{D(T)}^{\tau}$.

It remains to prove that $\Omega_{T}^{\tau}\subset\operatorname*{int}\,\!\!_{\tau}D(T)$.
Assume, by contradiction that $x\in\Omega_{T}^{\tau}\setminus\operatorname*{int}\,\!\!_{\tau}D(T)$.
Since, in this case, $x$ is a support point for the convex set $\overline{D(T)}^{\tau}$whose
interior is non-empty, we know that $N_{\overline{D(T)}^{\tau}}\{x\}\neq\{0\}$.
Together with $T=T+N_{\overline{D(T)}^{\tau}}$, this yields that
$Tx$ is unbounded; in contradiction to the fact that $Tx$ is equicontinuous.
\end{proof} 

\strut

Recall that a set $S\subset X$ is \emph{nearly-solid} if there is
a convex set $C$ such that $\operatorname*{int}C\neq\emptyset$ and
$C\subset S\subset\overline{C}$. Equivalently, $S$ is nearly-solid
iff $\operatorname*{int}S$ is non-empty convex and $S\subset\operatorname*{cl}(\operatorname*{int}S)$.
Indeed, directly, from $\operatorname*{int}C=\operatorname*{int}\overline{C}$
and $\operatorname*{cl}(\operatorname*{int}C)=\overline{C}$ (see
\cite[Lemma 11A b), p. 59]{MR0410335}) we know that $\operatorname*{int}S=\operatorname*{int}C$
is non-empty convex and $S\subset\overline{C}=\operatorname*{cl}(\operatorname*{int}S)$.
Conversely, $C=\operatorname*{int}S$ fulfills all the required conditions.

\begin{corollary} \label{tst} If $(X,\tau)$ is a barreled LCS and
$T\in\mathfrak{M}(X)$ has $\operatorname*{int}\,\!\!_{\tau}D(T)\neq\emptyset$
and $\overline{D(T)}^{\tau}$ convex then $D(T)$ is nearly-solid
and $\Omega_{T}^{\tau}=\operatorname*{int}\,\!_{\tau}D(T)$. \end{corollary}

\begin{theorem} \label{ot-bn} If $(X,\tau)$ is a barreled normed
space and $T\in\mathfrak{M}(X)$ has $\operatorname*{core}\operatorname*{Pr}\,\!\!_{X}(\operatorname*{dom}\varphi_{T})\neq\emptyset$
then $D(T)$ is nearly-solid and $\Omega_{T}^{\tau}=\operatorname*{core}\operatorname*{Pr}\,\!\!_{X}(\operatorname*{dom}\varphi_{T})=\operatorname*{int}\,\!_{\tau}D(T)$.
\end{theorem}

\begin{proof} It suffices to prove that $\overline{D(T)}^{\tau}$
is convex. We actually prove that
\[
\operatorname*{core}\operatorname*{Pr}\,\!_{X}(\operatorname*{dom}\varphi_{T})\subset D(T);
\]
from which $\overline{D(T)}^{\tau}=\operatorname*{cl}_{\tau}(\operatorname*{core}\operatorname*{Pr}\,\!_{X}(\operatorname*{dom}\varphi_{T}))=\operatorname*{cl}_{\tau}\operatorname*{Pr}\,\!_{X}(\operatorname*{dom}\varphi_{T})$
is convex. 

For a fixed $z\in\operatorname*{core}\operatorname*{Pr}\,\!_{X}(\operatorname*{dom}\varphi_{T})$
consider the function $\Phi:X^{*}\times X\rightarrow\overline{\mathbb{R}}$,
$\Phi(x^{*},x)=\varphi_{T}(x+z,x^{*})-\langle z,x^{*}\rangle$. Then
$0\in\operatorname*{core}(\operatorname*{Pr}_{X}(\operatorname*{dom}\Phi))$.
The function $\Phi$ is convex, proper, and $w^{*}\times\tau-$lsc
so it is also $s^{*}\times\tau-$lsc, where $s^{*}$ denotes the strong
topology on $X^{*}$. 

We may use \cite[Proposition\ 2.7.1\ (vi),\ p. 114]{MR1921556} to
get 
\begin{equation}
\inf_{x^{*}\in X^{*}}\Phi(x^{*},0)=\max_{y^{*}\in X^{*}}(-\Phi^{*}(0,y^{*})).\label{fd-1}
\end{equation}

Because $\varphi_{T}\ge c$ (see e.g. \cite[Corollary\ 3.9]{MR1009594}),
$\inf_{x^{*}\in X^{*}}\Phi(x^{*},0)\ge0$ and that $\Phi^{*}(x^{**},x^{*})=\varphi_{T}^{*}(x^{*},x^{**}+z)-\langle z,x^{*}\rangle$,
$x^{*}\in X^{*}$, $x^{**}\in X^{**}$. Therefore, from (\ref{fd-1}),
there exists $y^{*}\in X^{*}$ such that $\varphi_{T}^{*}(y^{*},z)=:\psi_{T}(z,y^{*})\le\langle z,y^{*}\rangle$,
that is, $(z,y^{*})\in[\psi_{T}\le c]=[\psi_{T}=c]=T$ (see \cite[Theorem\ 2.2]{MR2207807}
or \cite[Theorem\ 1]{MR2577332} for more details). Hence $z\in D(T)$.
\end{proof} 

\begin{corollary} \label{rock-1} If $X$ is a barreled normed space
and $T\in\mathfrak{M}(X)$ has $\operatorname*{core}(\operatorname*{conv}D(T))\neq\emptyset$
then $D(T)$ is nearly-solid and $\Omega_{T}=\operatorname*{int}D(T)$.
\end{corollary}

\begin{corollary} \label{rock-2} If $X$ is a Banach space and $T\in\mathfrak{M}(X)$
has $\operatorname*{int}(\operatorname*{conv}D(T))\neq\emptyset$
then $D(T)$ is nearly-solid and $\Omega_{T}=\operatorname*{int}D(T)$.
\end{corollary}

Under a Banach space settings the proof of the following result is
known from \cite[p. 406]{MR0253014} (see also \cite[Theorem 3.11.15, p. 286]{MR1921556})
and it relies on the Bishop-Phelps Theorem, namely, on the density
of the set of support points to a closed convex set in its boundary.
The novelty of our argument is given by the use of the maximal monotonicity
of the normal cone to a closed convex set.

\begin{theorem} \label{mnc}Let $X$ be a Banach space and let $T\in\mathfrak{M}(X)$
be such that $\overline{D(T)}$ is convex and $\Omega_{T}\neq\emptyset$.
Then $\Omega_{T}=\operatorname*{int}D(T)$. \end{theorem}

\begin{proof} For every $x\in\Omega_{T}$, let $U\in\mathscr{V}(x)$
be closed convex as in (\ref{cyber}). Since $T+N_{\overline{D(T)}}=T$
and $T(U)$ is equicontinuous we know that there are no support points
to $\overline{D(T)}$ in $U$, i.e., for every $u\in\overline{D(T)}\cap U=D(T)\cap U$,
$N_{\overline{D(T)}}|_{U}(u)=\{0\}$.

Then 
\[
N_{\overline{D(T)}\cap U}=N_{\overline{D(T)}}+N_{U}=N_{U}|_{D(T)}\subset N_{U},
\]
and since $N_{\overline{D(T)}\cap U}\in\mathfrak{M}(X)$, $N_{U}\in\mathcal{M}(X)$
we get that $U\subset D(T)$, and so $x\in\operatorname*{int}D(T)$.
\end{proof}

\strut

We conclude this paper with two results on the convex subdifferential.

\begin{theorem} \label{msd} Let $(X,\tau)$ be a LCS and let $f:X\to\overline{\mathbb{R}}$
be proper convex $\tau-$lower semicontinuous such that $f$ is continuous
at some $x_{0}\in\operatorname*{int}_{\tau}(\operatorname*{dom}f)$.
Then $\partial f\in\mathfrak{M}(X)$, $D(\partial f)$, $\operatorname*{dom}f$
are nearly-solid, $\operatorname*{int}_{\tau}(\operatorname*{dom}f)=\operatorname*{int}_{\tau}D(\partial f)$,
$\operatorname*{cl}_{\tau}(\operatorname*{dom}f)=\overline{D(\partial f)}^{\tau}$,
and $\Omega_{\partial f}^{\tau}=\operatorname*{int}\,\!\!_{\tau}\, D(\partial f)$.
\end{theorem}

\begin{proof} Recall that in this case $f$ is continuous on $\operatorname*{int}_{\tau}(\operatorname*{dom}f)$,
$\operatorname*{dom}f$ is a solid set, and $\operatorname*{int}_{\tau}(\operatorname*{dom}f)\subset D(\partial f)\subset\operatorname*{dom}f$
(see e.g. \cite[Theorems\ 2.2.9,\ 2.4.12]{MR1921556}). The latter
inclusions implies that $D(\partial f)$ is nearly-solid, and $\operatorname*{int}_{\tau}(\operatorname*{dom}f)=\operatorname*{int}_{\tau}D(\partial f)$,
$\operatorname*{cl}_{\tau}(\operatorname*{dom}f)=\overline{D(\partial f)}^{\tau}$. 

Clearly, $f(x)+f^{*}(x^{*})$ is a representative of $\partial f$
so $\partial f$ is representable. In order for $\partial f\in\mathfrak{M}(X)$
it suffices to prove that $\partial f$ is of negative infimum type
(see \cite[Theorem\ 1(ii)]{MR2577332} or \cite[Theorem\ 2.3]{MR2207807}).
Seeking a contradiction, assume that $(x,x^{*})\in X\times X^{*}$
satisfies $\varphi_{\partial f}(x,x^{*})<\langle x,x^{*}\rangle$;
in particular $(x,x^{*})$ is monotonically related to $\partial f$,
i.e., for every $(u,u^{*})\in\partial f$,$\langle x-u,x^{*}-u^{*}\rangle\ge0$.
Let $g:\mathbb{R}\to\overline{\mathbb{R}}$, $g(t):=f(tx+(1-t)x_{0})=f(Lt+x_{0})$,
where $L:\mathbb{R}\to X$, $Lt=t(x-x_{0})$. According to the chain
rule  
\[
\partial g(t)=L^{*}(\partial f(Lt+x_{0})),\ {\rm or}\ s\in\partial g(t)\Leftrightarrow\exists y^{*}\in\partial f(tx+(1-t)x_{0}),\ \langle x-x_{0},y^{*}\rangle=s.
\]
The function $g$ is proper convex lower semicontinuous so $\partial g\in\mathfrak{M}(\mathbb{R})$.
But, $(1,\langle x-x_{0},x^{*}\rangle)$ is monotonically related
to $\partial g$ because, for every $s\in\partial g(t)$ there is
$y^{*}\in\partial f(tx+(1-t)x_{0})$ such that $\langle x-x_{0},y^{*}\rangle=s$
which provides$(1-t)(\langle x,x^{*}\rangle-s)=\langle x-(tx+(1-t)x_{0}),x^{*}-y^{*}\rangle\ge0$.
Therefore $(1,\langle x-x_{0},x^{*}\rangle)\in\partial g$, in particular
$1\in D(\partial g)$ and $x\in D(\partial f)$. That implies the
contradiction $\varphi_{\partial f}(x,x^{*})\ge\langle x,x^{*}\rangle$. 

Hence $\varphi_{\partial f}(x,x^{*})\ge\langle x,x^{*}\rangle$, for
every $(x,x^{*})\in X\times X^{*}$, that is, $\partial f$ is of
negative infimum type and consequently maximal monotone. 

Also $\partial f$ is locally equicontinuous on $\operatorname*{int}_{\tau}(\operatorname*{dom}f)$
which shows that $\operatorname*{int}\,\!\!_{\tau}\, D(\partial f)\subset\Omega_{\partial f}^{\tau}$
(see e.g. \cite[Theorem\ 2.2.11]{MR1921556} and the proof of \cite[Theorem\ 2.4.9]{MR1921556}). 

Because $\partial f\in\mathfrak{M}(X)$, we see, as in the proof of
Theorem \ref{ot}, that $\Omega_{\partial f}^{\tau}\subset D(\partial f)$.
From $f=f+\iota_{\operatorname*{dom}f}$ we get $\partial f=\partial f+N_{\overline{D(\partial f)}}$
which shows that $\partial f(x)$ is unbounded for each $x\in D(\partial f)\setminus\operatorname*{int}_{\tau}D(\partial f)$
since $N_{\overline{D(\partial f)}}(x)\neq\{0\}$. Hence $\Omega_{\partial f}^{\tau}=\operatorname*{int}\,\!\!_{\tau}\, D(\partial f)$.
\end{proof}

\begin{theorem} \label{msd-bar} Let $(X,\tau)$ be a barreled LCS
and let $f:X\to\overline{\mathbb{R}}$ be proper convex $\tau-$lower
semicontinuous with $\operatorname*{core}(\operatorname*{dom}f)\neq\emptyset$.
Then $\partial f\in\mathfrak{M}(X)$, $D(\partial f)$, $\operatorname*{dom}f$
are nearly-solid, $\operatorname*{int}_{\tau}(\operatorname*{dom}f)=\operatorname*{int}_{\tau}D(\partial f)$,
$\operatorname*{cl}_{\tau}(\operatorname*{dom}f)=\overline{D(\partial f)}^{\tau}$,
and $\Omega_{\partial f}^{\tau}=\operatorname*{int}\,\!\!_{\tau}\, D(\partial f)$.
\end{theorem}

\begin{proof} It suffices to notice that under a barreled space context
$\operatorname*{core}(\operatorname*{dom}f)=\operatorname*{int}_{\tau}(\operatorname*{dom}f)$,
$f$ is continuous on $\operatorname*{int}_{\tau}(\operatorname*{dom}f)$
(see e.g. \cite[Theorem\ 2.2.20]{MR1921556}), and we may apply Theorem
\ref{msd}. \end{proof}

\bibliographystyle{plain}

\begin{thebibliography}{1}
\bibitem{MR1009594} Simon Fitzpatrick. \newblock Representing monotone
operators by convex functions. \newblock In {\em Workshop/{M}iniconference
on {F}unctional {A}nalysis and {O}ptimization ({C}anberra,
1988)}, volume~20 of {\em Proc. Centre Math. Anal. Austral. Nat.
Univ.}, pages 59--65. Austral. Nat. Univ., Canberra, 1988.

\bibitem{MR0410335} Richard~B. Holmes. \newblock {\em Geometric
functional analysis and its applications}. \newblock Springer-Verlag,
New York, 1975. \newblock Graduate Texts in Mathematics, No. 24.

\bibitem{MR0253014} R.~T. Rockafellar. \newblock Local boundedness
of nonlinear, monotone operators. \newblock {\em Michigan Math.
J.}, 16:397--407, 1969.

\bibitem{MR2207807} M.~D. Voisei. \newblock A maximality theorem
for the sum of maximal monotone operators in non-reflexive {B}anach
spaces. \newblock {\em Math. Sci. Res. J.}, 10(2):36--41, 2006.

\bibitem{MR3252437} M.~D. Voisei. \newblock The {M}inimal {C}ontext
for {L}ocal {B}oundedness in {T}opological {V}ector {S}paces.
\newblock {\em An. \c{S}tiin\c{t}. Univ. Al. I. Cuza Ia\c{s}i.
Mat. (N.S.)}, 59(2):219--235, 2013.

\bibitem{MR2577332} M.~D. Voisei and C.~Z{\u{a}}linescu. \newblock
Maximal monotonicity criteria for the composition and the sum under
weak interiority conditions. \newblock {\em Math. Program.}, 123(1,
Ser. B):265--283, 2010.

\bibitem{MR1921556} C.~Z{\u{a}}linescu. \newblock {\em Convex
analysis in general vector spaces}. \newblock World Scientific Publishing
Co. Inc., River Edge, NJ, 2002.\end{thebibliography}

\end{document}